\documentclass[11pt]{article}
\usepackage{amscd,amssymb,stmaryrd}
\usepackage[fleqn]{amsmath}
\usepackage{graphicx}
\usepackage{subcaption}
\usepackage{comment}
\usepackage[utf8]{inputenc}
\usepackage[export]{adjustbox}
\usepackage{wrapfig}
\usepackage{amstext}
\usepackage{pstricks,pst-node,pst-plot,pst-coil}
\usepackage{amsthm}
\usepackage{amsmath}
\usepackage{color}
\usepackage{mathtools}
\usepackage{hyperref}
\usepackage[english]{babel}
\usepackage{booktabs}
\usepackage{mathrsfs}
\usepackage{float} 
\usepackage{marginnote}
\usepackage{enumitem}

\newcommand\norm[1]{\left\|#1\right\|}
\newcommand\abs[1]{\lvert#1\rvert}

\newcommand{\mr}{\mathcal{R}}
\theoremstyle{definition}
\newtheorem{theorem}{Theorem}[section]

\newtheorem{lemma}[theorem]{Lemma}

\newtheorem{definition}[theorem]{Definition}
\theoremstyle{remark}
\newtheorem*{remark}{Remark}

\title{Computational multiscale methods for quasi-gas dynamic equations}
\author{Boris Chetverushkin\thanks{Keldysh Institute of Applied Mathematics (Russian Academy of Sciences), Moscow, Russia}\and Eric Chung\thanks{Department of Mathematics, The Chinese University of Hong Kong, Shatin, Hong Kong}\ , \quad Yalchin Efendiev\thanks{Department of Mathematics and Institute of Scientific Computing, Texas A\&M University, College Station, TX 77843, USA}\ , \quad Sai-Mang Pun\thanks{Department of Mathematics, Texas A\&M University, College Station, TX 77843, USA}\ , \quad and \quad Zecheng Zhang\thanks{Department of Mathematics, Texas A\&M University, College Station, TX 77843, USA}}
\begin{document}
\maketitle
\begin{abstract}
In this paper, we consider the quasi-gas-dynamic (QGD) model 
in a multiscale environment. The model equations can be regarded as
a hyperbolic regularization and are derived from kinetic equations.
So far, the research on QGD models has been focused on problems with constant
coefficients. In this paper, we investigate the QGD model in multiscale
media, which can be used in porous media applications. This multiscale problem
is interesting from a multiscale methodology point of view as the
model problem has a hyperbolic
multiscale term, and designing multiscale methods for hyperbolic equations
is challenging. In the paper, we apply the constraint energy minimizing generalized 
multiscale finite element method (CEM-GMsFEM) combined with 
the leapfrog scheme in time to solve this problem. 
The CEM-GMsFEM provides a flexible and systematical framework to construct crucial multiscale basis functions for approximating the solution to the problem with reduced computational cost. With this approach of spatial discretization, we establish the stability of the fully discretized scheme under a relaxed version of the so-called CFL condition. 
Complete convergence analysis of the proposed method is presented. Numerical results are provided to illustrate and verify the theoretical findings. 
\end{abstract}


\section{Introduction}


The simulations of complex flows play an important role in
 many applications, such
as porous media, aerodynamics, and so on. There are various model
equations used for
simulation purposes, which vary from kinetic to continuum models,
such as the Navier-Stokes equations. There are several intermediate-scale models
that are successfully used in the literature, which includes
the quasi-gas dynamic (QGD) system of equations. The QGD model has shown
to be effective for various applications.
The QGD model equations are derived from kinetic equations
under the assumption that the distribution function is similar 
to a locally Maxwellian representation. The QGD model has an advantage
that it guarantees the smoothing of the solution at the free path distance.
The QGD equations are extensively described in
the literature \cite{chetverushkin2005kinetic,chetverushkin2019compact,chetverushkin2016application,chetverushkin2018kinetic,chetverushkin2018hyperbolic,lutskii2019compact}.


In the paper, we consider a simplified 
QGD system involving second derivatives with respect to the time, in addition
to spatial diffusion. In literature, this model has also been used
to regularize purely parabolic equations by adding a hyperbolicity.
This regularization has been employed in designing efficient time
stepping algorithms \cite{chetverushkin2019compact,chetverushkin2018kinetic,chetverushkin2018hyperbolic}. 


We consider the QGD model in a multiscale environment. More
precisely, we consider a simplified QGD model (see \eqref{eqn:qgd})
and introduce multiscale coefficients. These coefficients represent
the media properties and spatially vary. The applications of these equations
can be considered in porous media for compressible flows. The heterogeneities
of the coefficients represent the media properties, which can have large 
variations. Our objective in this paper is to make some first steps in 
understanding multiscale systems in these hyperbolic quasi-dynamic systems.


In the paper, we would like to solve the QGD model equations on a coarse
grid that is much larger compared to spatial heterogeneities. There are
many methods for coarse-grid approximation. These include
 homogenization-based approaches \cite{cances2015embedded, chen2019homogenize,chen2019homogenization,fu2019edge,le2014msfem,le2014multiscale,salama2017flow}, 
multiscale
finite element methods \cite{hkj12,hw97,jennylt03,jennylt05}, 
generalized multiscale finite element methods \cite{MixedGMsFEM,WaveGMsFEM,chung2018fast,chung2015goal,GMsFEM13,gao2015generalized,chung2016adaptiveJCP}, constraint energy minimizing GMsFEM (CEM-GMsFEM) \cite{chung2018constraint, chung2018constraintmixed}, Nonlocal
Multi-continua Approaches (NLMC) \cite{NLMC},
metric-based upscaling \cite{oz06_1}, heterogeneous multiscale method \cite{abe07,ee03}, localized orthogonal decomposition (LOD) \cite{henning2012localized,maalqvist2014localization}, equation free approaches \cite{rk07,srk05,skr06}, computational continua \cite{fafalis2018computational,fish2010computational,fish2005multiscale}, hierarchical multiscale method \cite{brown2013efficient,hs05,tan2019high}, 
and so on. Some of these approaches, such as homogenization-based
approaches, are designed for problems with scale 
separation. In porous media applications, the spatial heterogeneities
are complex and do not have scale separation. In addition, they 
contain large jumps in the coefficients. As a result, the coarse grid
does not resolve scales and contrast. For these purposes, we have 
introduced a general concept CEM-GMSFEM and NLMC, where multiple basis
functions or continua are designed to solve problems on a coarse grid
\cite{chung2018constraintmixed, NLMC}. 
These approaches require a careful design of multiscale
basis functions. The applications of these methods to hyperbolic
equations are challenging \cite{WaveGMsFEM} due to distant temporal effects.
In this paper, our goal is to design an approach for hyperbolic quasi-dynamic systems.


For spatial discretization, we adopt the idea of CEM-GMsFEM presented in \cite{chung2018constraint} and construct a specific multiscale space for approximating the solution. 
Starting with a well-designed auxiliary space, we construct multiscale basis functions (supported in some oversampling regions) which are minimizers of a class of constraint energy minimization problems. 
One of the theoretical benefits of the CEM-GMsFEM is that the convergence of the method can be shown to be independent of the contrast from the heterogeneities; and the error linearly decreases with respect to coarse mesh size if the oversampling parameter is appropriately chosen. Our analysis indicates that a moderate number of oversampling layers, depending logarithmically on the contrast, seems sufficient to archive accurate approximation. 
The present CEM-GMsFEM setting allows flexibly adding additional basis functions based on spectral properties of the differential operators. This enhances the accuracy of the method in the presence of high contrast in the media. 
It is shown that if enough basis functions are selected in each local patch, the convergence of the method can be shown independently of the contrast. 

For temporal discretization, we use a central finite difference scheme to discretize the first and second order time derivatives in the equation. 
We show that the corresponding fully-discretized scheme is stable under a relaxed version of the CFL condition. 
In order to prove the stability and convergence of the full discretization, we first establish an inverse inequality in the multiscale finite element space. This result relies on the localized estimate between the global and local multiscale basis functions \cite{chung2018constraint}. 
A complete convergence analysis is presented in this work. In particular, the error estimate of semi-discretization is shown in Theorem \ref{semi_boss}. For the complete analysis of the fully-discretized numerical scheme, the main result is summarized in Theorem \ref{boss}. 
Throughout the part of analysis, we need proper regularity assumptions on the source term and initial conditions. 
Numerical results are provided to illustrate the efficiency of the proposed method and it confirms our theoretical findings. 

The remainder of the paper is organized as follows. We provide in Section \ref{sec:prelim} the background knowledge of the problem. Next, we introduce the multiscale method and the discretization in Section \ref{sec:method}. In Section \ref{sec:analysis}, we provide the stability estimate of the method and prove the convergence of the proposed method. 
We present the numerical results in Section \ref{sec:numerics}. Finally, we give concluding remarks in Section \ref{sec:conclusion}. 

\section{Preliminaries} \label{sec:prelim}
Consider the quasi-gas dynamics (QGD) model in a polygonal domain $\Omega \subset \mathbb{R}^d$ ($d = 2,3$): 
\begin{eqnarray} \label{eqn:qgd}
\begin{split}
u_t + \alpha u_{tt} -  \nabla \cdot ( \kappa \nabla u) & =  f \quad &\text{in } (0,T] \times \Omega, \\
u|_{t=0} & = u_0 \quad &\text{in } \Omega, \\
u_t|_{t=0} & = v_0 \quad &\text{in } \Omega, \\
u & = 0 \quad &\text{on } \partial \Omega. 
\end{split}
\end{eqnarray}
Here, $u_t$ denotes the time derivative of the function $u$, $\alpha$ is a constant, $\kappa: \Omega \to \mathbb{R}$ is a time-independent high-contrast permeability field such that $0 < \gamma \leq \kappa(x) \leq \beta$ for almost every $x \in \Omega$, $f$ is a source term with suitable regularity, and $T>0$ is the terminal time. 
Further, we assume that the initial conditions $u_0 \in H_0^1(\Omega)$ and $v_0 \in L^2(\Omega)$. 

We clarify the notation used throughout the work. We write $(\cdot,\cdot)$ to denote the inner product in $L^2(D)$ and $\norm{\cdot}$ for the corresponding norm. 
Let $H_0^1(\Omega)$ be the subspace of $H^1(\Omega)$ with functions having a vanishing trace and the corresponding dual space is denoted by $H^{-1}(\Omega)$. Moreover, we write $L^p(0,T;X)$ for the Bochner space with the norm
\begin{eqnarray*}
  \norm{v}_{L^p(0,T;X)} := \left (\displaystyle{\int_0^T \norm{v }_X^p dt} \right)^{1/p}  \quad 1\leq p < \infty, 
\end{eqnarray*}
where $X$ is a Banach space equipped with the norm $\norm{\cdot}_X$. 

Instead of the original PDE formulation, we consider the variational formulation corresponding to \eqref{eqn:qgd}: Find $u \in L^2(0,T; H_0^1(\Omega))$ with $u_t \in L^2(0,T; L^2(\Omega))$ and $u_{tt} \in L^2(0,T; L^2(\Omega))$ such that 
\begin{eqnarray} \label{eqn:qgd-var}
(u_t, v) + \alpha \left (u_{tt} ,v \right ) + a(u,v) = (f,v) 
\end{eqnarray}
for all $v \in V:= H_0^1(\Omega)$. 
Here, we define $a(u,v) := \int_\Omega \kappa \nabla u \cdot \nabla v~ dx$ for all $u, v \in V$. 
Employing Galerkin's method and the method of energy estimate, one can show the well-posedness of the variational formulation \eqref{eqn:qgd-var}. See \cite[Chapter 7.2]{evans2010partial} for more details. 

In this research, we apply the constraint energy minimizing generalized multiscale finite element method (CEM-GMsFEM) to approximate the solution of the above QGD model. First, we introduce fine and coarse grids for the computational domain.
Let $\mathcal{T}^H = \{ K_i \}_{i=1}^N$ be a conforming partition of the domain $\Omega$ with mesh size $H>0$ defined by
$$H := \max_{K \in \mathcal{T}^H} \Big(\max_{x, y \in K} \abs{x-y}\Big).$$ 
We refer to this partition as the coarse grid. We denote the total number of coarse elements as $N \in \mathbb{N}^+$. Subordinate to the coarse grid, we define the fine grid partition $\mathcal{T}^h$ (with mesh size $h \ll H$) by refining each coarse element $K \in \mathcal{T}^H$ into a connected union of finer elements. We assume that the refinement above is performed such that $\mathcal{T}^h$ is also a conforming partition of the domain $\Omega$. Denote $N_c$ the number of interior coarse grid nodes of $\mathcal{T}^H$ and we denote $\{ x_i \}_{i=1}^{N_c}$ the collection of interior coarse nodes in the coarse grid. 

\section{Multiscale method} \label{sec:method}
In this section, we outline the framework of CEM-GMsFEM and present the construction of the multiscale space for approximating the solution of the QGD model. We emphasize that the multiscale basis functions and the corresponding space are defined with respect to the coarse grid of the domain. The multiscale method consists of two steps. First, we perform a spectral decomposition and form an auxiliary space. 
Next, we construct a multiscale space for approximating the solution based on the auxiliary space.  We remark that these basis functions are locally supported in some coarse patches formed by some coarse elements. Once the multiscale spaces are ready, one can use leapfrog scheme to discretize time derivatives and solve the resulting fully-discretized problem. 

\subsection{The spectral decomposition}
We present the construction of the auxiliary multiscale basis functions. Let $K_i \in \mathcal{T}^H$ be a coarse block. Define $V(K_i)$ as the restriction of the abstract space $V$ on the coarse element $K_i$. We consider a local spectral problem: Find $\lambda_j^{(i)} \in \mathbb{R}$ and $\phi_j^{(i)} \in V(K_i)$ such that 
\begin{eqnarray} \label{eqn:spectral}
a_i(\phi_j^{(i)}, v) = \lambda_j^{(i)} s_i(\phi_j^{(i)},v ) \quad \text{for all } v \in V(K_i).
\end{eqnarray}
Here, $a_i: V(K_i) \times V(K_i)$ is a symmetric non-negative definite bilinear form and $s_i: V(K_i) \times V(K_i)$ is a symmetric positive definite bilinear form. 
We remark that the above problem is solved on a fine mesh in actual computations. Based on the analysis, we choose 
$$ a_i(v,w) := \int_{K_i} \kappa \nabla v \cdot \nabla w ~ dx, \quad s_i(v,w) := \int_{K_i} \tilde \kappa v w ~ dx , \quad \text{where} \quad \tilde \kappa := \sum_{j=1}^{N_c} \kappa \abs{\nabla \chi_j^{\text{ms}}}^2.$$
The functions $\{ \chi_j^{\text{ms}} \}_{j=1}^{N_c}$ are the standard multiscale finite element basis functions which satisfy the partition of unity property. More precisely, $\chi_j^{\text{ms}}$ is the solution of the following system: 
\begin{eqnarray*}
\nabla \cdot (\kappa \nabla \chi_j^{\text{ms}}) & = 0 \quad &\text{in each } K \subset \omega_j, \\
\chi_j^{\text{ms}} & = g_j \quad &\text{on } \partial K \setminus \partial \omega_j,\\
\chi_j^{\text{ms}} & = 0 \quad &\text{on } \partial \omega_j. 
\end{eqnarray*}
The function $g_j$ is continuous and linear along the boundary of the coarse element. We assume that the eigenvalues $\lambda_j^{(i)}$ are arranged in ascending order and we pick $\ell_i \in \mathbb{N}^+$ corresponding eigenfunctions to construct the local auxiliary space $V_{\text{aux}}^{(i)}:= \text{span} \{ \phi_j^{(i)} : j = 1, \cdots, \ell_i \}$. 
We assume the normalization $s_i\left ( \phi_j^{(i)}, \phi_j^{(i)} \right ) = 1$. 
After that, we define the global auxiliary multiscale space $V_{\text{aux}} := \bigoplus_{i=1}^{N} V_{\text{aux}}^{(i)}$. We remark that the global auxiliary multiscale space is used to construct multiscale basis functions that are orthogonal to the auxiliary space with respect to the weighted $L^2$ inner product $s(\cdot,\cdot)$. 

Note that the bilinear form $s_i(\cdot, \cdot)$ defines an inner product with norm $\norm{\cdot}_{s(K_i)} := \sqrt{s(\cdot,\cdot)}$ in the local auxiliary space $V_{\text{aux}}^{(i)}$. Based on these local inner products and norms, one can naturally define a new inner product and norm for the global auxiliary space $V_{\text{aux}}$ as follows: for all $v, w \in V_{\text{aux}}$, 
\begin{eqnarray}
    s(v,w) := \sum_{i=1}^N s_i(v,w) \quad \text{and} \quad \norm{v}_s := \sqrt{s(v,v)}.
    \label{l2norm}
\end{eqnarray}
The inner product and norm defined above can be extended for the abstract space $V$. 
Note that if $\{ \chi_j^{\text{ms}} \}_{j=1}^{N_c}$ is a set of bilinear partition of unity, then 
$\norm{v}_s \leq H^{-1} \beta^{1/2} \norm{v}$
for any $v \in L^2(\Omega)$. 
In addition, we define $\pi: L^2(\Omega) \to V_{\text{aux}}$ as the projection with respect to the inner product $s(\cdot,\cdot)$ such that 
$$ \pi u = \pi(u) := \sum_{i=1}^N \sum_{j=1}^{\ell_i} s_i (u, \phi_j^{(i)}) \phi_j^{(i)} \quad \text{for all } u \in L^2(\Omega). $$

\subsection{The construction of multiscale basis functions}
In this section, we present the construction of the multiscale basis functions. First, we define an oversampling region for each coarse element. Specifically, given a non-negative integer $m \in \mathbb{N}$ and a (closed) coarse element $K_i$, we define the oversampling region $K_{i,m} \subset \Omega$ such that 
$$ K_{i,m} := \left \{ \begin{array}{lr} 
K_i & \text{if } m = 0, \\
\displaystyle{\bigcup \{ K: K_{i,m-1} \cap K \neq \emptyset \}} & \text{if } m \geq 1.
\end{array} \right .$$

See Figure \ref{fig:mesh} for an illustration of oversampling region. For simplicity, we denote $K_i^+$ the oversampled region $K_{i,m}$ for some nonnegative integer $m$. 

\begin{figure}[ht]
\centering
\includegraphics[width = 2.5in]{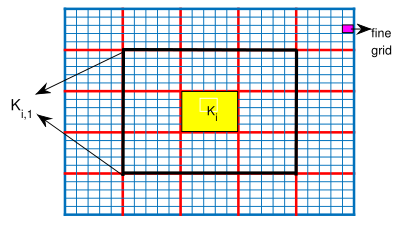}
\caption{Oversampling region with $m = 1$.} 
\label{fig:mesh}
\end{figure}

Recall that $V(K_i^+)$ is the restriction of $V$ on the coarse patch $K_i^+$. Let $V_0(K_i^+)$ be the subspace of $V(K_i^+)$ with zero trace on the boundary $\partial K_i^+$. 
For each eigenfunction $\phi_j^{(i)} \in V_{\text{aux}}$, we define the multiscale basis $\psi_{j,\text{ms}}^{(i)} \in V_0(K_i^+)$ to be the solution of the equation: 
\begin{eqnarray}
a(\psi_{j,\text{ms}}^{(i)}, v) + s\left(\pi( \psi_{j,\text{ms}}^{(i)}), \pi(v) \right )  =  s(\phi_j^{(i)},v) \quad \text{for all } v \in V_{0}(K_i^+) \label{eqn:msv}.
\end{eqnarray}
Then, the multiscale space is defined as 
$ V_{\text{ms}} := \text{span} \left \{ \psi_{j,\text{ms}}^{(i)}: i = 1,\cdots, N, ~ j = 1,\cdots, \ell_i \right \}$. 
By construction, we have $\text{dim}(V_{\text{ms}}) = \text{dim}(V_{\text{aux}})$. 

\begin{remark}
The local construction of multiscale basis function $\psi_{j, \text{ms}}^{(i)}$ supported in $K_i^+$ is motivated by the following global construction: Find $\psi_j^{(i)} \in V$ such that 
\begin{eqnarray}
a(\psi_{j}^{(i)}, v) + s\left(\pi( \psi_{j}^{(i)}), \pi(v) \right )  =  s(\phi_j^{(i)},v) \quad \text{for all } v \in V \label{eqn:gmsv}. 
\end{eqnarray}
We then define $V_{\text{glo}} := \text{span} \left \{ \psi_{j}^{(i)}: i = 1,\cdots, N, ~ j = 1,\cdots, \ell_i \right \}$. It has been shown in \cite{chung2018constraint} that the decomposition $V = V_{\text{glo}} \oplus \text{Ker}(\pi)$ holds and this decomposition is orthogonal with respect to the energy bilinear form $a(\cdot,\cdot)$. We will use this property to prove the inverse inequality (Lemma \ref{lem:inver-ineq}) below. 

Using the result of \cite[Lemma 5]{chung2018constraint}, we have the error estimate of localization: For any multiscale function $v_{\text{ms}} = \sum_{i=1}^N \sum_{j=1}^{\ell_i} \alpha_j^{(i)} \psi_{j,\text{ms}}^{(i)} \in V_{\text{ms}}$, there exists a function $v_{\text{glo}} = \sum_{i=1}^N \sum_{j=1}^{\ell_i} \alpha_j^{(i)} \psi_{j}^{(i)} \in V_{\text{glo}}$ such that 
\begin{eqnarray}
\norm{v_{\text{glo}} - v_{\text{ms}}}_a \lesssim (m+1)^{d/2} E^{1/2} \sum_{i=1}^N \sum_{j=1}^{\ell_i} \left ( \alpha_j^{(i)} \right)^2. 
\label{eqn:exp-decay}
\end{eqnarray}
Here, $m$ is the number of oversampling, $E := 3(1+ \Lambda^{-1})\left ( 1+ (2(1+\Lambda^{-1/2})) \right)^{1-m}$ is the factor of exponential decay, and $\Lambda := \displaystyle{\min_{1 \leq i \leq N} \lambda_{\ell_i +1}^{(i)}}$ with $\left \{ \lambda_j^{(i)} \right \}$ being obtained from \eqref{eqn:spectral}. 
\end{remark}

\subsection{The method and discretization}  \label{discretization}
In this section, we discuss the discretizations of the equation \eqref{eqn:qgd-var}. 
Let $u_{\text{ms}} \in V_{\text{ms}}$ be the multiscale approximation to the exact solution $u$. In particular, the function $u_{\text{ms}}$ solves 
\begin{eqnarray} \label{eqn:qgd-var-ms}
\left ( (u_{\text{ms}})_t, v \right) + \alpha \left ( (u_{\text{ms}})_{tt} ,v \right ) + a(u_{\text{ms}},v) = (f,v) \quad \text{for all } v \in V_{\text{ms}}.
\end{eqnarray}
For time discretization, we first partition the temporal domain $(0,T)$ into equally $N_T$ pieces with time step size $\Delta t$. For any function $v = v(t)$, we use the following finite differences to approximate time derivatives appearing in the QGD model: 
$$ v_t \approx \frac{v(t_{n+1}) - v(t_{n-1})}{2\Delta t} =: D_t v^n\quad \text{and} \quad v_{tt} \approx \frac{v(t_{n+1}) - 2v(t_n) + v(t_{n-1})}{(\Delta t)^2} =: D_{tt} v^n.$$
The fully discretization of the equation \eqref{eqn:qgd-var} reads: Find $\mathbf{u}_{H}^{T} := \left ( u_{H}^n \right )_{n=0}^{N_T}$ with $u_{H}^n \in V_{\text{ms}}$ such that for any $n = 1, \cdots, N_{T}-1$, 
\begin{eqnarray} \label{eqn:fully-dis}
\left (D_t u_H^n +  \alpha D_{tt} u_H^n,v \right ) + a(u_{H}^n,v) = (f^n,v) \quad \text{for all } v \in V_{\text{ms}},
\end{eqnarray}
where $f^n := f(t_n)$. 

\section{Convergence analysis} \label{sec:analysis}
In this section, we analyze the convergence of the multiscale method. 
Throughout the work, we denote $a \lesssim b$ if there is a generic constant $C>0$ such that $a \leq C b$. We write $a \lesssim_T b$ if there is a constant $C_T$ depending on $T$ such that $a \leq C_T b$. We denote $\norm{\cdot} = \norm{\cdot}_{L^2(\Omega)}$ and $\norm{\cdot}_a := \sqrt{a(\cdot,\cdot)}$.

\subsection{Semi-discretized scheme} 

We first consider the stability and error estimate in semi-discretization. The following results give a stability estimate for the scheme \eqref{eqn:qgd-var-ms}. 

\begin{lemma} \label{lem:qgd-semi-stable}
Let $u_{\text{ms}} \in V_{\text{ms}}$ be the solution of the equation \eqref{eqn:qgd-var-ms}. Then, 
\begin{eqnarray} \label{eqn:stability_semi}
\alpha \norm{(u_{\text{ms}})_t(T)}^2 + \norm{(u_{\text{ms}})(T)}_a^2  \lesssim   \alpha \norm{v_0}^2 + \norm{u_0}_a^2 + \norm{f}_{L^2(0,T; L^2(\Omega))}^2 .
\end{eqnarray}
\end{lemma}
\begin{proof}
Let $v = (u_{\text{ms}})_t$ in \eqref{eqn:qgd-var-ms}. We have 
$$ \norm{(u_{\text{ms}})_t}^2 + \frac{1}{2} \frac{d}{dt}\left ( \alpha \norm{(u_{\text{ms}})_t}^2 + \norm{u_{\text{ms}}}_a^2 \right ) = (f,(u_{\text{ms}})_t)  \leq \norm{f} \cdot \norm{(u_{\text{ms}})_t}.$$
We remark that if $ f \equiv 0$, the scheme is of energy conservation. Integrating over $(0,T)$ leads to 
\begin{equation*}
\begin{split}
2 \int_0^T \norm{(u_{\text{ms}})_t}^2 dt +\alpha \norm{(u_{\text{ms}})_t (T)}^2  + \norm{u_{\text{ms}}(T) }_a^2 & \leq \alpha \norm{v_0}^2 + \norm{u_0}_a^2 + 2 \int_0^T \frac{1}{\sqrt{2}}\norm{f} \cdot \sqrt{2} \norm{(u_{\text{ms}})_t} dt \\
& \leq \alpha \norm{v_0}^2 + \norm{u_0}_a^2 +  \frac{1}{2}  \int_0^T\norm{f}^2 dt + 2 \int_0^T  \norm{(u_{\text{ms}})_t}^2 dt
\end{split}
\end{equation*}
using Cauchy-Schwarz inequality. This completes the proof. 
\end{proof}

To estimate the error bound for semi-discretization scheme, we introduce the definition of elliptic projection. 
\begin{definition}
For any function $v \in V$, we define the elliptic projection $\widehat{v} \in V_{\text{ms}}$ of the function $v$ such that 
\begin{eqnarray}\label{eqn:elliptic-project}
a(v - \widehat{v}, w ) = 0 \quad \text{for all } w \in V_{\text{ms}}.
\end{eqnarray}
\end{definition}

Next, we analyze the convergence of the proposed multiscale method. For any function $v \in V$, we define the {\it energy functional} $\mathcal{E}: V \to \mathbb{R}$ such that 
$ \mathcal{E}(v) := \sqrt{\alpha} \norm{v_t} + \norm{v}_a$. 
It is not difficult to verify that 
$$ \mathcal{E}( v+w ) = \sqrt{\alpha} \norm{v_t + w_t} + \norm{v+w}_a \leq \sqrt{\alpha} \left ( \norm{v_t} + \norm{w_t} \right ) + \norm{v}_a + \norm{w}_a = \mathcal{E}(v) + \mathcal{E}(w)$$
for any $v, w \in V$. That is, the triangle inequality holds for the energy functional. Note that for any $v \in V$, we have 
$$ \left ( \mathcal{E}(v) \right)^2 = \left ( \sqrt{\alpha} \norm{v_t} + \norm{v}_a \right)^2 \lesssim   \alpha \norm{v_t}^2 + \norm{v}_a^2.$$

We have the following error estimate for the semi-discretization of the QGD model. 

\begin{theorem}
Let $u \in V$ be the solution to \eqref{eqn:qgd-var} and $u_{\text{ms}} \in V_{\text{ms}}$ be the multiscale solution to \eqref{eqn:qgd-var-ms}. Assume that the number of oversampling layers $m = O(\log (\beta \gamma ^{-1} H^{-1}))$ and $\{\chi_j^{\text{ms}} \}_{j=1}^{N_c}$ are bilinear partition of unity. 
Then, for any $t\in (0,T]$, the following error estimate holds 
\begin{eqnarray}
\norm{u(t) - u_{\text{ms}}(t)}_a \lesssim_T H \Lambda^{-1/2},
\label{eqn:semi_dis}
\end{eqnarray}
where $\Lambda = \displaystyle{\min_{1 \leq i \leq N} \lambda_{\ell_i +1}^{(i)}}$ and $\{ \lambda_j^{(i)} \}$ are the eigenvalues obtained by solving \eqref{eqn:spectral}. 
\label{semi_boss}
\end{theorem}
\begin{proof}
Denote $\widehat{u}$ the elliptic projection of the exact solution $u$. We write 
$$ e:= u - u_{\text{ms}} = \underbrace{ u - \widehat{u} }_{=: \rho} + \underbrace{ \widehat{u} - u_{\text{ms}} }_{=: \theta} = \rho + \theta. $$
Denote $\mathcal{F} := f - u_t - \alpha u_{tt}$. Note that the function $\widehat{u}$ satisfies the equation: 
$$ a(\widehat{u} , v) = ( \mathcal{F}, v) \quad \text{for all } v \in V_{\text{ms}}.$$
Using the result of \cite[Lemma 1]{chung2018constraint}, we obtain that 
$$ \norm{\rho}_a = \norm{u - \widehat{u}}_a \lesssim H \Lambda^{-1/2} \norm{\kappa^{-1/2} \mathcal{F}} \quad \text{and} \quad \norm{\rho_t} = \norm{\left (u - \widehat{u}\right)_t} \lesssim H^2 \Lambda^{-1} \norm{\kappa^{-1/2}\mathcal{F}_t}.$$
Therefore, we have 
$$ \mathcal{E}(\rho) \lesssim \sqrt{\alpha} H^2 \Lambda^{-1} \norm{\kappa^{-1/2}\mathcal{F}_t} + H \Lambda^{-1/2} \norm{\kappa^{-1/2} \mathcal{F}} \lesssim H\Lambda^{-1/2}. $$
Next, we analyze the term $\mathcal{E}(\theta)$. Subtracting \eqref{eqn:qgd-var-ms} from \eqref{eqn:qgd-var}, we obtain 
$$ (e_t, v) + \alpha \left (e_{tt}, v \right )+ a(e, v) = 0 \quad \text{for all } v \in V_{\text{ms}}.$$
Note that, by the property of elliptic projection, we have $a(\rho, v) = 0$ for all $v \in V_{\text{ms}}$. 
That is, we have 
$$ (\theta_t, v) + \alpha \left ( \theta_{tt}, v \right) + a(\theta ,v) = \left ( ( \widehat{u} - u)_t + \alpha (\widehat{u} - u)_{tt}, v \right )$$ for all $v \in V_{\text{ms}}$. 
Denote $\mathcal{G} := ( \widehat{u} - u)_t  + \alpha ( \widehat{u} - u)_{tt}$. Let $v = \theta_t \in V_{\text{ms}}$ and use the same technique for proving the stability result \eqref{eqn:stability_semi}, one can show that 
$$ \left ( \mathcal{E} (\theta) \right )^2  \lesssim \alpha \norm{\theta_t (0)}^2 + \norm{\theta (0)}_a^2 + \norm{\mathcal{G}}_{L^2(0,T; L^2(\Omega))}^2.$$
Note that $\theta_t(0)$ and $\theta(0)$ are given by the initial conditions of quasi gas-dynamics equation. If we choose $u_{\text{ms}}(0)$ be such that 
$$ a( u_{\text{ms}}(0), v) = a(u_0, v) \quad \text{for all } v \in V_{\text{ms}},$$
then $\theta_t (0) = \theta (0) = 0$ because of the property of elliptic projection. Therefore, we have 
$$ \mathcal{E}(\theta) \lesssim \norm{\mathcal{G}}_{L^2(0,T; L^2(\Omega))} \lesssim \norm{\rho_t}_{L^2(0,T;L^2(\Omega))} + \alpha \norm{\rho_{tt}}_{L^2(0,T;L^2(\Omega))} \lesssim_T H^2 \Lambda^{-1}.$$
To conclude, we show that 
\begin{eqnarray} \label{eqn:semi-estimate}
\mathcal{E} ( u - u_{\text{ms}}) \leq \mathcal{E} (\rho) + \mathcal{E}(\theta) \lesssim_T H\Lambda^{-1/2}.
\end{eqnarray}
This completes the proof. 
\end{proof}

\subsection{Fully discretization}
In this section, we analyze the method in fully discretization. 
First, we define $\sigma_{\text{aux}} :=\displaystyle{ \max_{1 \leq i \leq N} \left ( \max_{1 \leq j \leq \ell_i} \lambda_j^{(i)}\right)}$. We observe that the inverse inequality (in the multiscale space) holds. To prove the inverse inequality in $V_{\text{ms}}$, we first prove the following lemma.
\begin{lemma}
For any $v_{\text{ms}} = \sum_{i=1}^N \sum_{j=1}^{\ell_i} \alpha_j^{(i)} \psi_{j,\text{ms}}^{(i)} \in V_{\text{ms}}$, the following estimation holds
\begin{eqnarray}
\sum_{i=1}^N \sum_{j=1}^{\ell_i} (\alpha_j^{(i)})^2 \leq (1+D) \norm{v_{\text{ms}}}_s^2, 
\label{alpha_ij}
\end{eqnarray}
where $D$ is a generic constant depending on the value of $\sigma_{\text{aux}}$.
\end{lemma}
\begin{proof}
Let $v_{\text{ms}} = \sum_{i = 1}^{N}\sum_{j = 1}^{\ell_i}\alpha_j^{(i)}\psi_{j,\text{ms}}^{(i)}\in V_{\text{ms}}$. 
By the variational formulation \eqref{eqn:msv}, for any $\phi_k^{(l)}\in V_{\text{aux}}$, we have
$$
s(\pi v_{\text{ms}}, \phi_{k}^{(l)} ) = \sum_{i = 1}^{N}\sum_{j = 1}^{\ell_i}\alpha_j^{(i)}s(\pi\psi_{j, \text{ms}}^{(i)}, \phi_k^{(l)})    = \sum_{i = 1}^{N}\sum_{j = 1}^{\ell_i}\alpha_j^{(i)}
\left ( s(\pi\psi_{j, \text{ms}}^{(i)}, \pi\psi_{k,\text{ms}}^{(l)}) +a(\psi_{j, \text{ms}}^{(i)}, \psi_{k, \text{ms}}^{(l)})   \right ). 
$$
Denote $b_{lk} = s(\pi v_{\text{ms}}, \phi_k^{(l)})$ and $\textbf{b} = (b_{lk})$, we have
$$
\|\textbf{c}\|_2\leq\|A^{-1}\|_2\cdot\|\textbf{b}\|_2,
$$
where $A\in \mathbb{R}^{p\times p}$ is the matrix representation of the bilinear form 
$$
s(\pi \psi_{j, \text{ms}}^{(i)}, \pi\psi_{k, \text{ms}}^{(l)} )+a(\psi_{j, \text{ms}}^{(i)}, \psi _{k,\text{ms}}^{(l)} )
$$
with $p = \sum_{i = 1}^N \ell_i$ and $ \textbf{c} = \left (\alpha_j^{(i)} \right )\in \mathbb{R}^p$. We then estimate the largest eigenvalue of $A^{-1}$. Define an auxiliary function $\phi := \sum_{i = 1}^{N}\sum_{j = 1}^{\ell_i} \alpha_j^{(i)}\phi_j^{(i)} \in V_{\text{aux}}$ and $\psi_{\text{ms}} \in V_{\text{ms}}$ to be the solution of the following equation: 
\begin{eqnarray}
a(\psi_{\text{ms}}, \omega) +s(\pi\psi_{\text{ms}}, \pi\omega) = s(\phi, \pi\omega) \quad  \text{for all } \omega\in V_{\text{ms}}. 
\label{eqn:var_inverse}
\end{eqnarray}
On the other hand, by \cite[Lemma 2]{chung2018constraint}, there is a function $z\in V$ such that 
$$\pi z = \phi \quad \text{and}  \quad \|z\|_a^2 \leq D \|\phi\|_s^2. $$ 
Here, $D$ is a generic constant depending on the value of $\sigma_{\text{aux}}$ (cf. \cite[Lemma 2]{chung2018constraint}). 
Taking $\omega = z$ in (\ref{eqn:var_inverse}) and using the fact that $s(\phi, \phi) = \|\textbf{c}\|_2^2$, we have 
\begin{eqnarray*}
\begin{split}
\|\textbf{c}\|_2^2 & = a(\psi_{\text{ms}}, z )+s(\pi \psi_{\text{ms}}, \phi)
\leq \|\psi_{\text{ms}}\|_a\|z\|_a+\|\pi\psi_{\text{ms}}\|_s\|\phi\|_s \\
& \leq (1+D)^{\frac{1}{2}}\|\phi\|_s \left (\|\psi_{\text{ms}}\|_a^2+\|\pi\psi_{\text{ms}}\|_s^2 \right )^{\frac{1}{2}}.
\end{split}
\end{eqnarray*}
This implies that $\|A^{-1}\|_2\leq (1+D)^{\frac{1}{2}}$. It follows that
$ \|\textbf{c}\|_2^2\leq (1+D) \|\textbf{b}\|_2^2\leq (1+D) \|v_{\text{ms}}\|_s^2$.
\end{proof}

\begin{lemma}[Inverse Inequality] \label{lem:inver-ineq}
Assume that $\{ \chi_j^{\text{ms}} \}_{j=1}^{N_c}$ is a set of bilinear partition of unity. For any $v_{\text{ms}} \in V_{\text{ms}}$, there is a constant $C_{\text{inv}}>0$ such that
\begin{eqnarray} \label{eqn:inv-ineq}
\norm{\nabla v_{\text{ms}}} \leq C_{\text{inv}} H^{-1} \norm{v_{\text{ms}}}.
\end{eqnarray}
\label{inverse_ineq}
\end{lemma}
\vspace{-1cm}
\begin{proof}
Let $v \in V_{\text{glo}}$. Applying the orthogonality of $V_{\text{glo}}$, we get
\begin{eqnarray*}
\begin{split}
\gamma \norm{\nabla v}^2 & \leq a(v,v) = a(v, \pi v) \leq \norm{v}_a \norm{\pi v}_a \\
& \leq \beta^{1/2} \norm{\nabla v} \sigma_{\text{aux}}^{1/2} \norm{\pi v}_s \\
& \leq \beta \norm{\nabla v} \sigma_{\text{aux}}^{1/2} \norm{\pi v}_s.
\end{split}
\end{eqnarray*}
This implies that $\norm{\nabla v} \leq \gamma^{-1} \beta \sigma_{\text{aux}}^{1/2} \norm{\pi v}_s$ for any $v\in V_{\text{glo}}$.

Next, for any $v_{\text{ms}} = \sum_{i=1}^N \sum_{j=1}^{\ell_i} \alpha_j^{(i)} \psi_{j,\text{ms}}^{(i)}\in V_{\text{ms}}$, let $v = \sum_{i=1}^N \sum_{j=1}^{\ell_i} \alpha_j^{(i)} \psi_{j}^{(i)}\in V_{\text{glo}}$. 
We claim that $ \norm{\pi v}_s \leq \sum_{i=1}^N \sum_{j=1}^{\ell_i} (\alpha_j^{(i)})^2$. Notice that by \eqref{eqn:gmsv}, we have
\begin{eqnarray*}
\begin{split}
\norm{\pi v}_s^2 &= s(\pi v, \pi v)
= \sum_{i=1}^N \sum_{j=1}^{\ell_i} \alpha_j^{(i)}  s(\pi \psi_{j}^{(i)}, \pi v)
=  \sum_{i=1}^N \sum_{j=1}^{\ell_i} \alpha_j^{(i)} \left( s(\phi_{j}^{(i)}, \pi v) - a(\psi_{j}^{(i)},v) \right) \\
& = s(\phi, \pi v) - a (v, v) = s(\phi, \pi v) - \norm{v}_a^2
\end{split}
\end{eqnarray*}
with $\phi := \sum_{i=1}^N \sum_{j=1}^{\ell_i} \alpha_j^{(i)} \phi_j^{(i)}$. 
This implies that 
$$
\norm{\pi v}_s^2 \leq s(\phi, \pi v) \leq \norm{\phi}_s \norm{\pi v}_s \implies \norm{\pi v}_s \leq \norm{\phi}_s = \sum_{i=1}^N \sum_{j=1}^{\ell_i} \left (\alpha_j ^{(i)} \right)^2
$$
using the orthogonality of the auxiliary basis functions. By the inequalities \eqref{eqn:exp-decay} and  \eqref{alpha_ij}, we have 
\begin{eqnarray*}
\begin{split}
    \norm{\nabla v_{\text{ms}}}^2 &\leq \norm{\nabla (v-v_{\text{ms}})}^2+\norm{\nabla v}^2\\
    & \lesssim \gamma^{-1}  (m+1)^dE \sum_{i=1}^N \sum_{j=1}^{\ell_i} \left (\alpha_j^{(i)} \right )^2+
    \gamma^{-1} \beta \sigma_{\text{aux}}^{1/2} \norm{\pi v}_s \\
    & \lesssim   \left ( \gamma^{-1}  (m+1)^d E + \gamma^{-1} \beta \sigma_{\text{aux}}^{1/2} \right ) \sum_{i=1}^N \sum_{j=1}^{\ell_i} \left (\alpha_j^{(i)} \right )^2\\
    & \lesssim   \left ( \gamma^{-1}  (m+1)^d E + \gamma^{-1} \beta \sigma_{\text{aux}}^{1/2} \right )  (1+D) \norm{v_{\text{ms}}}_s^2. 
\end{split}
\end{eqnarray*}
Using the definition of $s$-norm, this gives that $\norm{\nabla v_{\text{ms}}} \leq C_{\text{inv}} H^{-1}\norm{v_{\text{ms}}}$ holds for any $v_{\text{ms}} \in V_{\text{ms}}$ with $C_{\text{inv}} =  \beta^{1/2}(1+D) \left ( \gamma^{-1}  (m+1)^d E + \gamma^{-1} \beta \sigma_{\text{aux}}^{1/2} \right )$. 
\end{proof}

Recall that $\mathbf{u}_{H}^{T} := \left ( u_H^n \right )_{n=0}^{N_T}$ with $u_H^n \in V_{\text{ms}}$ is the solution to \eqref{eqn:fully-dis}. The following result gives the stability estimate of the fully discretization. 

\begin{lemma} [Stability of the method]
Assume that the CFL condition 
\begin{eqnarray}
\alpha - \frac{1}{2} \beta C_{\text{inv}}^2 H^{-2} (\Delta t)^2 \geq \delta
\label{eqn:cfl}
\end{eqnarray}
holds for some constant $\delta >0$. Then, the fully discretization method \eqref{eqn:fully-dis} is stable; that is, 
\begin{eqnarray} \label{eqn:stability_fully}
\alpha \norm{\frac{u_H^n - u_H^{n-1}}{\Delta t}} + \norm{u_H^n}_a \lesssim \left ( \Delta t \sum_{k=1}^n \norm{f^k} + \alpha \norm{\frac{u_H^1 - u_H^0}{\Delta t}} + \norm{u_H^1}_a + \norm{u_H^0}_a\right).
\end{eqnarray}
\end{lemma}

\begin{proof}
Let $v = u_H^{n+1} - u_H^{n-1}$ in \eqref{eqn:fully-dis}. We have 
\begin{eqnarray*}
\begin{split}
\frac{1}{2\Delta t} \norm{u_H^{n+1} - u_H^{n-1}}^2 + \frac{\alpha}{(\Delta t)^2} \left ( u_H^{n+1} - u_H^n - (u_H^n - u_H^{n-1}), u_H^{n+1} - u_H^n + u_H^n - u_H^{n-1} \right) \\+ a ( u_H^n , u_H^{n+1} - u_H^{n-1}) = \Delta t \left (f^n,\frac{u_H^{n+1} - u_H^{n-1}}{\Delta t}\right). 
\end{split}
\end{eqnarray*}
Define $\mathcal{E}_{n,H} := \displaystyle{\frac{1}{2} \left ( \alpha \norm{\frac{u_H^{n} - u_H^{n-1}}{\Delta t}}^2 + a(u_H^{n-1}, u_H^n) \right )}$. 
It implies that 
$$ \alpha \left ( \norm{\frac{u_H^{n+1} - u_H^n}{\Delta t}}^2 -  \norm{\frac{u_H^{n} - u_H^{n-1}}{\Delta t}}^2\right ) + a(u_H^n , u_H^{n+1} ) - a(u_H^{n-1}, u_H^n) \leq (f^n , u_H^{n+1} - u_H^{n-1})$$
$$ \implies \mathcal{E}_{n+1,H} \leq \mathcal{E}_{n,H} + \frac{1}{2} (f^n , u_H^{n+1} - u_H^{n-1}).$$

Note that 
\begin{equation*}
\begin{split}
\mathcal{E}_{n,H} & = \frac{1}{2} \left ( \alpha \norm{\frac{u_H^n - u_H^{n-1}}{\Delta t}}^2 + a(u_H^n , u_H^{n-1}) \right ) \\
& = \frac{\alpha}{2} \norm{\frac{u_H^n - u_H^{n-1}}{\Delta t}}^2 + \frac{1}{4}  a(u_H^n, u_H^n) + \frac{1}{4} a(u_H^{n-1}, u_H^{n-1}) - \frac{1}{4} a(u_H^n - u_H^{n-1}, u_H^n - u_H^{n-1}) \\
& \geq \frac{\alpha}{2} \norm{\frac{u_H^n - u_H^{n-1}}{\Delta t}}^2 + \frac{1}{4}  a(u_H^n, u_H^n) + \frac{1}{4} a(u_H^{n-1}, u_H^{n-1}) - \frac{1}{4} \beta \norm{\nabla (u_H^n - u_H^{n-1})}^2 \\
& \geq \frac{\alpha}{2} \norm{\frac{u_H^n - u_H^{n-1}}{\Delta t}}^2 + \frac{1}{4}  a(u_H^n, u_H^n) + \frac{1}{4} a(u_H^{n-1}, u_H^{n-1}) - \frac{1}{4} \beta C_{\text{inv}}^2 H^{-2} (\Delta t)^2 \norm{\frac{u_H^n - u_H^{n-1}}{\Delta t}}^2\\
& = \frac{1}{2} \left ( \alpha  - \frac{1}{2} \beta C_{\text{inv}}^2 H^{-2} (\Delta t)^2 \right ) \norm{\frac{u_H^n - u_H^{n-1}}{\Delta t}}^2 + \frac{1}{4} \left ( \norm{u_H^n}_a^2 + \norm{u_H^{n-1}}_a^2 \right ). 
\end{split}
\end{equation*}

Then, we have 
\begin{eqnarray*}
\begin{split}
\mathcal{E}_{n+1,H}-  \mathcal{E}_{n,H}  & \leq \frac{1}{2} (f^n, u_H^{n+1} - u_H^{n-1}) \leq \frac{1}{2} \Delta t\norm{f^n} \left ( \norm{\frac{u_H^{n+1} - u_H^{n}}{\Delta t}} + \norm{\frac{u_H^{n} - u_H^{n-1}}{\Delta t}}\right ) \\
& \leq \frac{1}{2} \Delta t \norm{f^n} \cdot \sqrt{\frac{2}{\delta}} \left ( \sqrt{\mathcal{E}_{n+1,H}} + \sqrt{\mathcal{E}_{n,H}} \right ), \\
\sqrt{\mathcal{E}_{n+1,H}} - \sqrt{\mathcal{E}_{n,H}} & \leq \frac{1}{\sqrt{2\delta}} \Delta t \norm{f^n} \implies \sqrt{ \mathcal{E}_{n,H}}  \leq \sqrt{\mathcal{E}_{0,H}} + \frac{\Delta t}{\sqrt{2\delta}} \sum_{k=1}^n \norm{f^k}.
\end{split}
\end{eqnarray*}

This implies that
\begin{equation*}
\alpha \norm{\frac{u_H^n - u_H^{n-1}}{\Delta t}} + \norm{u_H^n}_a \lesssim  \left ( \Delta t \sum_{k=1}^n \norm{f^k} + \alpha \norm{\frac{u_H^1 - u_H^0}{\Delta t}} + \norm{u_H^1}_a + \norm{u_H^0}_a\right).
\end{equation*}
This completes the proof. 
\end{proof}

Recall that $u \in V$ is the solution of \eqref{eqn:qgd-var}. The total error between $\mathbf{u} := \left ( u(t_n) \right )_{n=0}^{N_T}$ and $\mathbf{u}_{H}^T$ can be split into two parts: the spatial discretization error $u(t_n) - u_{\text{ms}}(t_n)$ and the time discretization error $u_{\text{ms}}(t_n) - u_H^n$. Here, $u_{\text{ms}} \in V_{\text{ms}}$ is the solution of \eqref{eqn:qgd-var-ms}. Using the result of \eqref{eqn:semi-estimate}, we have 
$$ \norm{u(t_n) - u_{\text{ms}}(t_n)}_a \lesssim_T H\Lambda^{-1/2}.$$
Next, we estimate the time discretization error. Let $\widetilde e_n := u_{\text{ms}}^n- u_H^n$ with $u_{\text{ms}}^n := u_{\text{ms}}(t_n)$. Subtracting \eqref{eqn:qgd-var-ms} from \eqref{eqn:fully-dis}, we obtain 
$$ \left ( \frac{\widetilde e_{n+1} - \widetilde e_{n-1}}{2 \Delta t}, v\right ) + \alpha \left ( \frac{\widetilde e_{n+1} - 2\widetilde e_n + \widetilde e_{n-1}}{(\Delta t)^2}, v \right ) + a(\widetilde e_n, v) = \left ( \mathcal{H}^n ,v \right ) \quad \text{for all } v \in V_{\text{ms}},$$
where 
$$ \mathcal{H}^n := (u_{\text{ms}})_t + \alpha (u_{\text{ms}})_{tt} - \frac{u_{\text{ms}}^{n+1} - u_{\text{ms}}^{n-1}}{2\Delta t} - \alpha \frac{u_{\text{ms}}^{n+1} - 2u_{\text{ms}}^n + u_{\text{ms}}^{n-1}}{(\Delta t)^2}.$$
Using the result of \eqref{eqn:stability_fully}, one can obtain 
\begin{equation} \label{eqn:fully-conv}
\begin{split}
\alpha \norm{\frac{\widetilde e_{n+1} - \widetilde e_n}{\Delta t}} + \norm{\widetilde e_n}_a & \lesssim \alpha \norm{\frac{\widetilde e_1 - \widetilde e_0}{\Delta t}} + \norm{\widetilde e_1}_a + \Delta t \sum_{k=1}^n \left \{ \norm{(u_{\text{ms}})_t - \frac{u_{\text{ms}}^{k+1} - u_{\text{ms}}^{k-1}}{2\Delta t}} \right . \\
& \quad \left . + \alpha \norm{(u_{\text{ms}})_{tt} - \frac{u_{\text{ms}}^{k+1} - 2u_{\text{ms}}^k + u_{\text{ms}}^{k-1}}{(\Delta t)^2}} \right \}.
\end{split}
\end{equation}
Under the assumption of some additional regularity and appropriate initial conditions, the right-hand side of \eqref{eqn:fully-conv} scales like $H + (\Delta t)^2$.

Finally, we have the error estimate for the fully discretization scheme. 

\begin{theorem}
Assume that $u$, $u_{\text{ms}}$, and $f$ are smooth enough with respect to the variable $t$. Let $\widetilde u_{H} (t)$ be the piecewise linear function that interpolates $\mathbf{u}_{H}^T$ in time. Then 
$$ \norm{u - \widetilde u_{H,\text{ms}}}_{L^2(0,T; a)} \lesssim_T H + (\Delta t)^2, \quad \text{where } \norm{\cdot}_{L^2(0,T;a)} := \left ( \int_0^T \norm{\cdot}_a^2 dt \right )^{1/2}.$$
\label{boss}
\end{theorem}

\section{Numerical experiments} \label{sec:numerics}
In this section, we present several numerical experiments to demonstrate the efficiency of the proposed method. 
We set the computational domain $\Omega = (0,1)^2$. We partition the domain into $100 \times 100$ rectangular elements and refer it as a fine mesh $\mathcal{T}^h$ with mesh size $h = \sqrt{2}/100$. 

In the example below, we solve the QGD model \eqref{eqn:qgd-var} with $f (x_1, x_2) = \sin(\pi x_1)\sin(\pi x_2)$. Terminal time $T = 4.0$ is set and step size $\Delta t$ is chosen subjected to the CFL condition. 
The initial conditions are $u_0 = v_0 = 0$. 
Practical experiments showed that $\Delta t = 10^{-5}$ provides a sufficient and rather sharp choice for the stability with small value of $\alpha$ and high value of contrast. 
To implement the scheme, we set $u_{H}^0 = u_{H}^1 = 0$.
We use the permeability field $\kappa$ with contrast $10^3$ (see Figure \ref{kappa}). 
 \begin{figure}[h]
\centering
\includegraphics[scale = 0.4]{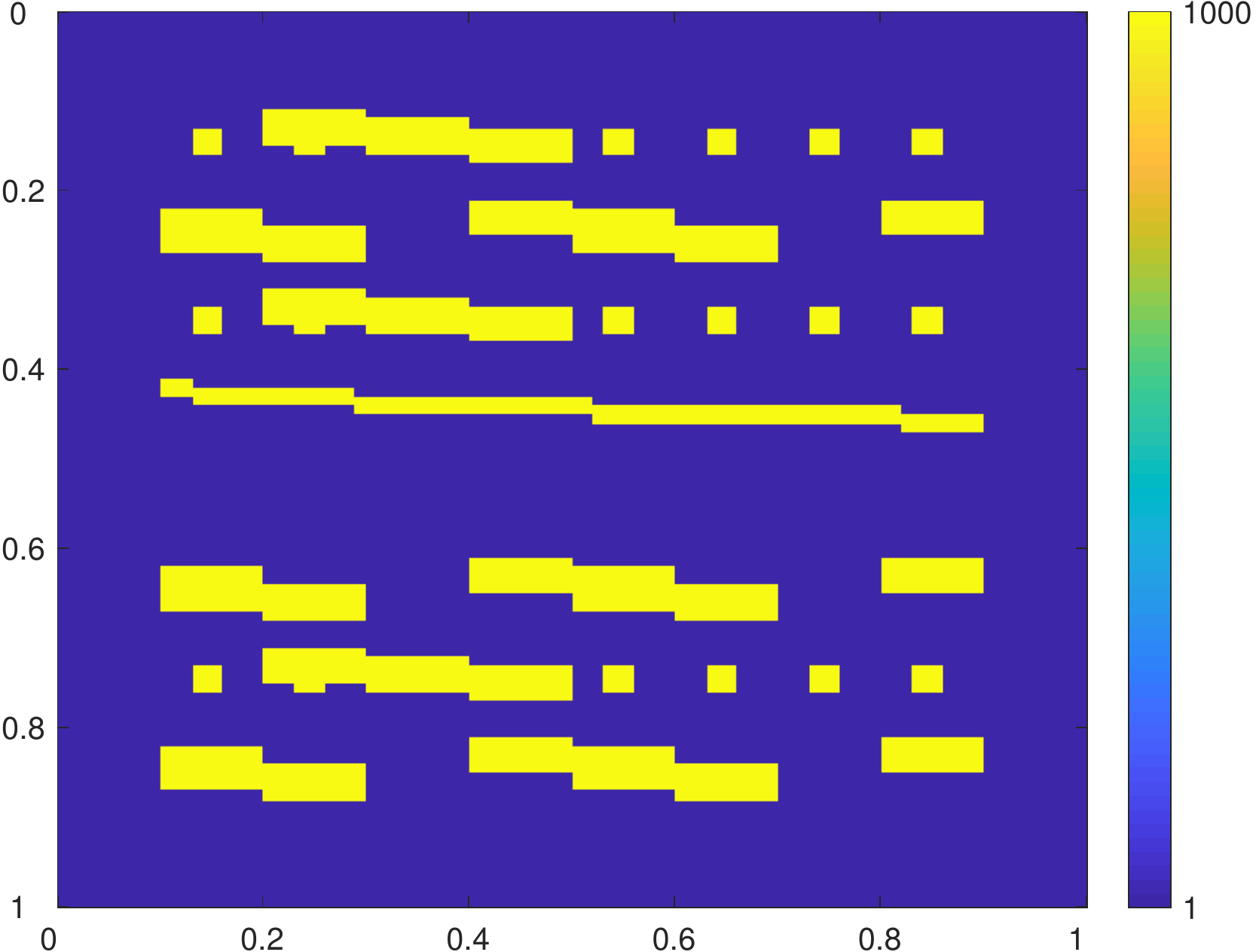}
\caption{Permeability field $\kappa$ with contrast values $10^{3}$.}
\label{kappa}
\end{figure}

We solve the fully discretization \eqref{eqn:fully-dis} and seek $u_{H}^n \in V_{\text{ms}}$. We define the 
corresponding relative $L^2$ and energy errors between the multiscale solution and the exact solution (up to a fine-scale) as follows: 
$$
    e_{L^2} := \frac{\|u(T) - u_{H}^{N_T} \|_s}{\|u(T)\|_s} \quad \text{and} \quad 
    e_a := \frac{\|u(T) - u_{H}^{N_T} \|_a}{\|u(T)\|_a},
$$    
 where $\norm{\cdot}_a = \sqrt{a(\cdot,\cdot)}$ and $\norm{\cdot}_s = \sqrt{s(\cdot,\cdot)}$. 

We present the convergence history in the energy and $L^2$ norms when the coarse mesh size is $H = \sqrt{2}/5$, $\sqrt{2}/10$, and $\sqrt{2}/20$, respectively. 
The number of oversampling layers $m$ is set to be $3,4$, and $6$ in all experiments. The number of multiscale basis functions is $\ell_i = 3$ in each local coarse element $K_i$. 
We test with different values of $\alpha = 0.001, 0.005, 0.1, 0.5, 1, 5,$ and $10$.
The results of $e_{L^2}$ and $e_a$ are shown in Tables \ref{table:1} and \ref{table:2}, respectively. A first-order convergence in energy norm and second-order convergence in $L^2$ norm have been observed as expected; see Figure \ref{l2a01} for illustration. 
 
 \begin{figure}[H]
\centering
\mbox{
\includegraphics[scale = 0.4]{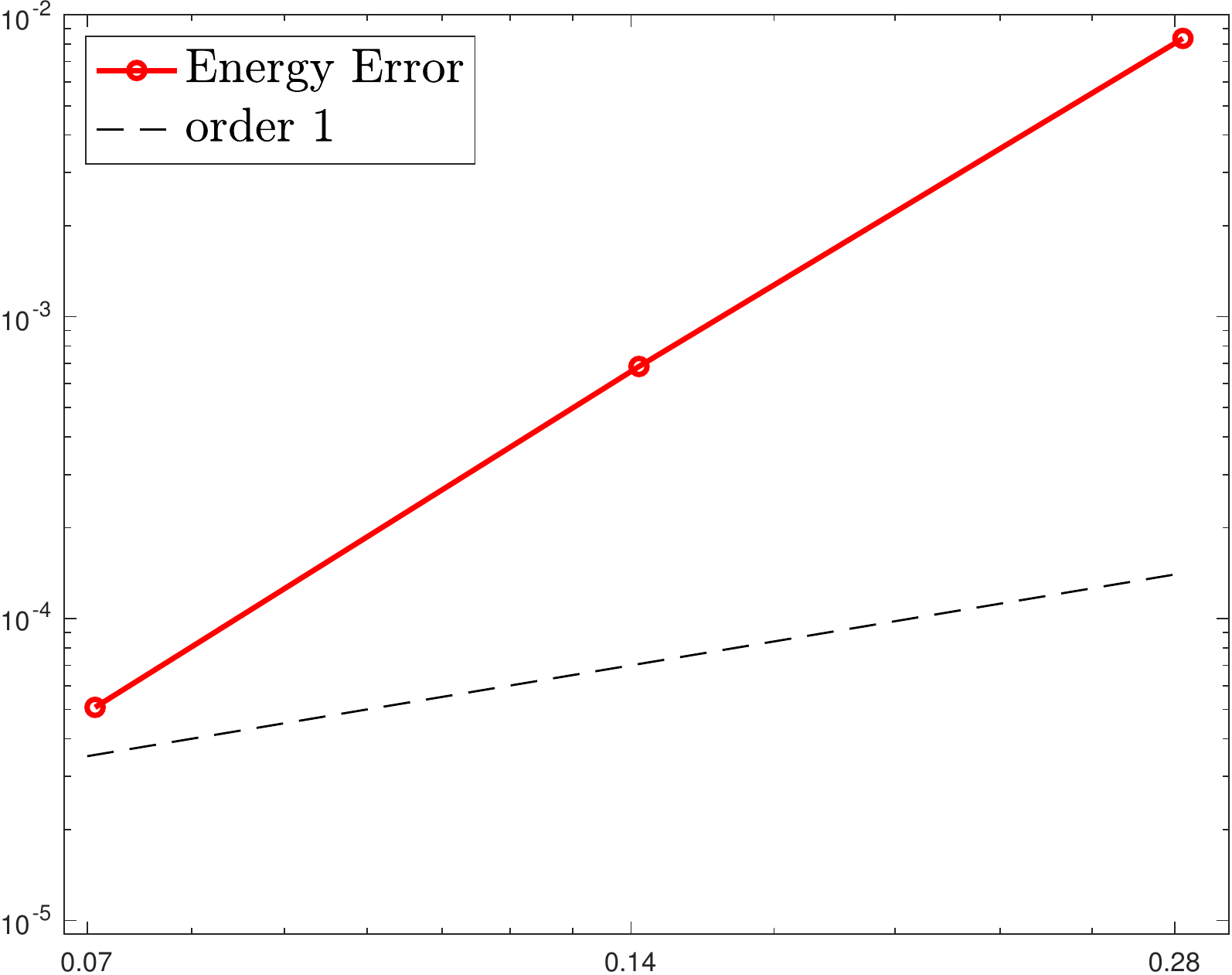}

\includegraphics[scale = 0.4]{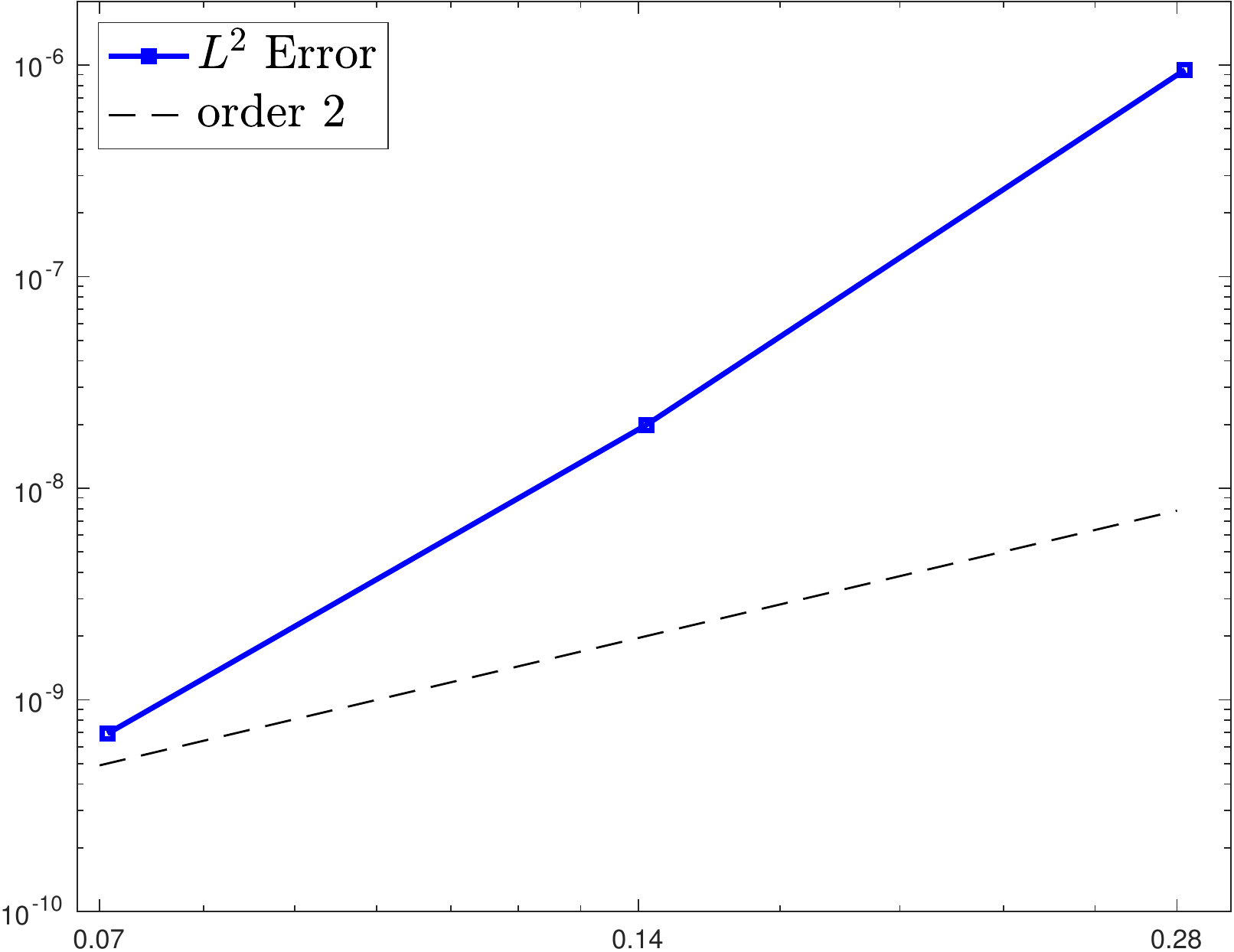}
}
\caption{Convergence history in $e_a$ (left) and $e_{L^2}$ (right) with $\alpha = 0.1$.}
\label{l2a01}
\end{figure}

\begin{table}[H]
\centering
\begin{tabular}{||c c c c c c c c c||} 
\hline
$H$ & $m$ &$\alpha = 10$ & $\alpha = 5$ & $\alpha = 1$ & $\alpha = 0.5$ & $\alpha = 0.1$ & $\alpha = 0.05$ & $\alpha = 0.01$ \\ [0.5ex] 
\hline
$\sqrt{2}/5$ & 3 & 2.07e-03 & 4.85e-05 & 2.09e-05 & 5.40e-06 & 9.49e-07 & 9.49e-07 & 9.49e-07 \\ 
\hline
$\sqrt{2}/10$ & $4$
& 9.39e-06 & 2.12e-07 & 1.60e-07 & 4.17e-08 & 1.99e-08 & 1.99e-08 & 1.99e-08 \\
\hline
$\sqrt{2}/20$ & $6$
& 1.95e-07 & 5.38e-09 & 2.45e-09 & 6.92e-10 & 6.92e-10 & 6.92e-10 & 6.92e-10 \\[1ex] 
\hline

\end{tabular}
\caption{Convergence in relative $L^2$ norm for different $\alpha$}
\label{table:1}
\end{table}
 
\begin{table}[H]
\centering
\begin{tabular}{||c c c c c c c c c||} 
\hline
$H$ & $m$ & $\alpha = 10$ & $\alpha = 5$ & $\alpha = 1$ & $\alpha = 0.5$ & $\alpha = 0.1$ & $\alpha = 0.05$ & $\alpha = 0.01$ \\ [0.5ex] 
\hline

$\sqrt{2}/5$  & 3 & 2.08e-02 & 8.76e-03 & 8.54e-03 & 8.55e-03 & 8.53e-03 & 8.53e-03 & 8.53e-03 \\ 
\hline
$\sqrt{2}/10$ & $4$
& 1.75e-03 & 6.28e-04 & 6.74e-04 & 6.86e-04 & 6.84e-04 & 6.84e-04 & 6.84e-04 \\
\hline
$\sqrt{2}/20$ & $6$
& 1.89e-04 & 5.19e-05 & 5.11e-05 & 5.09e-05 & 5.08e-05 & 5.08e-05 & 5.08e-05 \\[1ex] 
\hline
\end{tabular}
\caption{Convergence in relative energy norm for different $\alpha$}
\label{table:2}
\end{table}

We also test our algorithm on a problem with time dependent source. In this example, we set $f (x_1, x_2, t) = \sin(\pi t)\sin(\pi x_1)\sin(\pi x_2)$. All the other settings are same with the first example. The convergence in $L_2$ and energy norm are presented in Tables \ref{table:3} and \ref{table:4}. Convergence rate in both norms are observed. 
\begin{table}[H]
\centering
\begin{tabular}{||c c c c c c c c c||} 
\hline
$H$ & $m$ & $\alpha = 10$ & $\alpha = 5$ & $\alpha = 1$ & $\alpha = 0.5$ & $\alpha = 0.1$ & $\alpha = 0.05$ & $\alpha = 0.01$ \\ [0.5ex] 
\hline

$\sqrt{2}/5$  & 3 & 3.00e-01 & 7.83e-01 & 1.89e-01 & 7.41e-03 & 7.57e-03 & 7.11e-03 & 6.76e-03 \\ 
\hline
$\sqrt{2}/10$ & $4$
& 1.07e-03 & 3.41e-03 & 8.82e-04 & 4.80e-05 & 4.40e-05 & 4.16e-05 & 3.98e-05 \\
\hline
$\sqrt{2}/20$ & $6$
& 1.03e-05 & 2.70e-05 & 6.52e-06 & 3.46e-07 & 3.17e-07 & 3.00e-07 & 2.87e-07 \\[1ex] 
\hline
\end{tabular}
\caption{Convergence (time dependent source) in relative $L_2$ norm for different $\alpha$}
\label{table:3}
\end{table}

\begin{table}[H]
\centering
\begin{tabular}{||c c c c c c c c c||} 
\hline
$H$ & $m$ & $\alpha = 10$ & $\alpha = 5$ & $\alpha = 1$ & $\alpha = 0.5$ & $\alpha = 0.1$ & $\alpha = 0.05$ & $\alpha = 0.01$ \\ [0.5ex] 
\hline

$\sqrt{2}/5$  & 3 & 2.0198 & 1.5981 & 1.0128 & 0.8306 & 0.8304 & 0.8299 & 0.8295 \\ 
\hline
$\sqrt{2}/10$ & $4$
& 0.0656 & 0.0589 & 0.0565 & 0.0557 & 0.0558 & 0.0558 & 0.0558 \\
\hline
$\sqrt{2}/20$ & $6$
& 0.0072 & 0.0048 & 0.0048 & 0.0048 & 0.0048 & 0.0048 & 0.0048 \\[1ex] 
\hline
\end{tabular}
\caption{Convergence (time dependent source) in relative energy norm for different $\alpha$}
\label{table:4}
\end{table}

\section{Concluding remarks} \label{sec:conclusion}
In this work, we have proposed a novel computational multiscale method based on the idea of constraint energy minimization for solving the problem of quasi-gas-dynamics. The spatial discretization is based on CEM-GMsFEM which provides a framework to systematically construct multiscale basis functions for approximating the solution of the model. The multiscale basis functions with locally minimal energy are constructed by employing the techniques of oversampling, which leads to an improved accuracy in the simulations. Combined with the central difference scheme for the time discretization, we have shown that the fully discrete method is stable under a relaxed version of CFL condition and has optimal convergence rates despite the heterogeneities of the media. Numerical results have been presented to illustrate the performance of the proposed method.

\section*{Acknowledgement}

The research of Eric Chung is partially supported by the Hong Kong RGC General Research Fund (Project numbers 14304719 and 14302018) and the CUHK Faculty of Science Direct Grant 2019-20.

\bibliographystyle{abbrv}
\bibliography{references}
\end{document}